\documentclass[12pt]{article}

\usepackage{amsmath,amsthm,amssymb,amsfonts}
\usepackage{latexsym}
\newtheorem{theorem}{Theorem}
\newtheorem{remark}{Remark}
\newtheorem{lemma}{Lemma}
\newtheorem{corollary}{Corollary}
\newtheorem{proposition}{Proposition}
\begin{document}
\title{Forced Convex Mean Curvature Flow in Euclidean Spaces}
\author{Guanghan Li$^{\dag\ddag}$
 \and Isabel Salavessa$^{\ddag}$}
\date{}
\protect\footnotetext{\!\!\!\!\!\!\!\!\!\!\!\!\! {\bf MSC 2000:}
Primary: 53C44, 35K55; Secondary: 53A05.
 The first author is partially supported by NSFC (No.10501011) and
by Funda\c{c}\~{a}o Ci\^{e}ncia
e Tecnologia (FCT) through a FCT fellowship SFRH/BPD/26554/2006.
 The second author is partially supported by FCT through
the Plurianual of CFIF and POCI/MAT/60671/2004.}
\maketitle ~~~\\[-7mm]
{\footnotesize $\dag$ School of Mathematics and Computer
Science, Hubei University, Wuhan, 430062, \\[-1mm]P. R. China.
email: liguanghan@163.com}\\
{\footnotesize $\ddag$
  Centro de F\'{\i}sica das Interac\c{c}\~{o}es Fundamentais,
Instituto Superior T\'{e}cnico,\\[-1mm]
Technical University of Lisbon,
Edif\'{\i}cio Ci\^{e}ncia, Piso 3,
Av.\ Rovisco Pais, \\[-1mm] 1049-001 Lisboa, Portugal.
email: isabel.salavessa@ist.utl.pt}

\begin{abstract}
In this paper, we consider the mean curvature flow of convex
hypersurfaces in Euclidean spaces with a general forcing term. We
show that the flow may shrink to a point in finite time if the
forcing term is small, or exist for all times and expand to infinity
if the forcing term is large enough. The flow can also converge to a
round sphere for some special forcing term and initial hypersurface.
Furthermore, the normalization of the flow is carried out so that
long time existence and convergence of the rescaled flow are
studied. Our work extends Huisken's well-known mean curvature flow
and McCoy's mixed volume preserving mean curvature flow.
\end{abstract}
\section{Introduction}
\label{introduction}

\renewcommand{\thesection}{\arabic{section}}
\renewcommand{\theequation}{\thesection.\arabic{equation}}
\setcounter{equation}{0}

\indent

Let $M^n$ be a smooth and compact manifold of dimension $n\geq 2$
without boundary, and $X_0:M^n\longrightarrow {\mathbb R}^{n+1}$
be a smooth hypersurface immersion of $M^n$ which is strictly
convex. We consider a smooth family of maps $X_t=X(\cdot, t)$
evolving according to
\begin{eqnarray}\left\{\begin{array}{lll}
\frac {\partial }{\partial t} X(x,t) &=& \{h(t)-H(x, t)\}{\bf
v}(x,t), \quad x\in M^n,\\[2mm]
X(\cdot, 0)&=& X_0,
\end{array}\right.
\end{eqnarray}
where $H$ is the mean curvature of $M_t=X_t(M^n)$, ${\bf v}$ the
outer unit normal vector field, and $h(t)$ a nonnegative continuous
function. The curvature flow (1.1) is a strictly parabolic equation
and the short time existence easily follows from \cite {hp}.
Therefore we suppose that the evolution equation (1.1) has a smooth
solution on a maximal time interval $[0, T_{\max})$ for some
$T_{\max}>0$. Often different forcing term will lead to different
maximal time interval. We always assume that $h(t)$ is continuous in
$[0, T_{\max})$.

If $h(t)=0$, (1.1) is just the well-known mean curvature flow
\cite {h1}. In this case, (1.1) is contracting and $T_{\max}$ is
finite. If $h(t)$ is the average of the mean curvature on $M_t$,
i.e. $h(t)= {\int _{M_t}H_td\mu_t}/{\int _{M_t}d\mu_t}$, where
$d\mu_t$ is the area element of $M_t$, (1.1) is then the volume
preserving mean curvature flow \cite {h2}, which exists on all
time $[0, \infty)$, and the solution converges to a round sphere.
The hypersurfaces area preserving mean curvature flow for which
$h(t)=\int_{M_t}H_t^2d\mu_t/\int_{M_t}H_td\mu_t$ also exists for
all time and converges to a round sphere \cite {m1}. The mixed
volume preserving mean curvature flow \cite {m2} for which
$h(t)={\int _{M_t}HE_{k+1}d\mu_t }/{\int _{M_t}E_{k+1}d\mu_t}$,
$k=-1, 0, 1, \cdots, n-1$, where $E_l$ is the $l$-th elementary
symmetric function of the principal curvatures of $M_t$,
generalizes the results of the volume preserving mean curvature
flow \cite {h2} and surfaces area preserving mean curvature flow
\cite {m2}, and exists for all time and converges to a round
sphere. In fact, it can be checked that if the forcing term $h$ is
a small constant, the solution to (1.1) is still contracting. But
if $h$ is large enough, the curvature flow (1.1) expands and the
solution exists for all time.

From above, we see that different forcing term $h(t)$ leads to
different existence and convergence. A natural question is how to
unify all these cases?

In this paper, we study the curvature flow (1.1) with a general
forcing term $h(t)$ such that the limit $\lim_{t\rightarrow
T_{\max}} h(t)$ exists. We want to show that if the initial
hypersurface is convex and compact, the shape of $M_t$ approaches
the shape of a round sphere as $t\rightarrow T_{\max}$. In order to
describe the shape of the limiting hypersurface, we carry out a
normalization as in \cite {h1}. For any time $t$, where the solution
$X(\cdot, t)$ of (1.1) exists, let $\psi (t)$ be a positive factor
such that the hypersurface $\widetilde{M}_t$ given by
 $$\widetilde{X}(x, t)=\psi (t)X(x,t)$$
  has total area equal to
$|M_0|$, the area of $M_0$
$$\int_{\widetilde{M}_t}d\widetilde{\mu}_t=|M_0|, \qquad \mbox {for
all } t\in [0, T_{\max}).$$

After choosing the new time variable $\tilde{t}(t)=\int
_0^t\psi^2(\tau)d\tau$, we will see that $\widetilde{X}$ satisfies
the following evolution equation
\begin{eqnarray} \left\{\begin{array}{l}\frac{\partial }
{\partial \tilde{ t}} \widetilde{X}
=\{\widetilde{h}-\widetilde{H}\}{\bf \widetilde{v}}+{\frac 1n}
\widetilde{\theta}\widetilde{X},
\\[2mm]
\widetilde{X}(\cdot, 0)= X_0, \end{array}\right.
\end{eqnarray}
where $\widetilde{h}=\psi ^{-1}h$, $\widetilde{\theta}=\psi
^{-2}\theta$ and $\theta$ is given by
$$\theta=-\frac {\int _{M}(h-H)Hd\mu}{\int _Md\mu}.$$

In section 3, we have a time sequence $\{T_i\}$ such that
$T_i\rightarrow T_{\max}$ as $i\rightarrow \infty$, and a limit
$$\lim_{T_{i}\rightarrow T_{\max}}\psi (T_i)=\Lambda.$$

We now state our main theorem:
\begin{theorem}Let $n\geq 2$ and $M_0$  an $n$-dimensional smooth,
compact and strictly convex
 hypersurface immersed in ${\mathbb R}^{n+1}$.
Then for any nonnegative continuous function $h(t)$,
 there exists a unique, smooth solution to the evolution equation $(1.1)$ on a
 maximal time interval $[0, T_{\max})$. If additionally the following
limit exists and satisfies
\begin{equation}
\lim_{t\rightarrow T_{\max}}h(t)=\overline h< +\infty,
\end{equation}
  then we have:\\[-3mm]

 $(I)$ If $\Lambda =\infty$, then $T_{\max}<\infty$ and the curvature flow 
$(1.1)$ converges uniformly to a point as $t\rightarrow T_{\max}$.
Moreover the normalized equation $(1.2)$ has a solution
 $\widetilde{X}(x,\tilde{t})$ for all times $0\leq \tilde{t}< \infty$, and
 the hypersurfaces $\widetilde{M}(x, \tilde{t})$
  converge to a round sphere of area $|M_0|$ in the $C^{\infty}-$topology,
  as $\tilde{t}\rightarrow \infty$.\\[-4mm]

  $(II)$ If $0<\Lambda<\infty$, then $T_{\max}=\infty$, and the solutions
to $(1.1)$  converge uniformly to a round sphere in the $C^{\infty}-$topology
 as  $t\rightarrow  \infty$.\\[-4mm]

  $(III)$ If $\Lambda =0$, then $T_{\max}=\infty$. Moreover if
$\overline h\neq 0$, the solutions
  to $ (1.1)$ expand uniformly to $\infty$ as $t\rightarrow \infty$ and
   if the rescaled solutions to $(1.2)$ converge to a smooth
  hypersurface, then the limit must be a round sphere of total area $|M_0|$.
\end{theorem}

\begin{remark}
$(i)$ One can check that Theorem 1 includes Huisken's mean curvature
flow \cite {h1} and volume preserving mean curvature flow \cite
{h2}, McCoy's surface area preserving mean curvature flow \cite
{m1} and mixed volume preserving mean curvature flow \cite {m2}.\\[1mm]
$(ii)$ The assumption $(1.3)$ seems not natural since often the maximal
existing time $T_{\max}$ of $(1.1)$ depends on $h(t)$. In fact we can
use a stronger assumption that $h(t)$ is a nonnegative continuous
function on $[0, \infty)$ and satisfies $\lim_{t\rightarrow
\infty}h(t)< +\infty$. Our result still includes all cases in $(i)$.
\end{remark}

The extreme cases of Theorem 1 can also be considered.
\begin{remark}
$(i)$ For case $(I)$, when $\overline h=\infty$, $T_{\max}$ may not be
finite, even though $M_t$ is contracting (see Remark 3 $(ii)$ in
section 4). A sphere: $r(t)=\frac{1}{t+1}$,
$h(t)=n(t+1)-\frac{1}{(t+1)^2}$, is such an example, whose maximal
existing time $T_{\max}=\infty$.\\[1mm]
$(ii)$ For case $(III)$, if $\overline h=0$, $T_{\max}$ is also
infinite (see section 6). We don't know whether the solutions to
$(1.1)$ expand uniformly to $\infty$ as $t\rightarrow \infty$, but
we can find the special solution satisfying that condition. In
fact, a sphere: $r(t)=\sqrt{t+1}$,
$h(t)=\frac{2n+1}{2\sqrt{t+1}}$, is such a particular example, for
which $M_t$ expands to infinity. If $\overline h=\infty$, by
similar discussion as in section 6, we can show that $M_t$ expands
to infinity, but $T_{\max}$ may not be $\infty$. For example, the
sphere $r(t)=\frac{1}{1-t}$, $h(t)=n(1-t)+\frac{1}{(1-t)^2}$ is a
solution to (1.1), for which $T_{\max}=1$, and $r\rightarrow
\infty$, as $t\rightarrow 1$.
\end{remark}

We remark that Curvature flow in Euclidean spaces with different
forcing terms $h(t)$ were also studied by
Schn$\ddot{u}$rer-Smoczyk \cite {ss}, and Liu-Jian \cite {lj1}. If
the ambient space is a Minkowski space, Aarons \cite {aa} studied
the forced mean curvature flow of graphs and obtained the long
time existence and convergence under suitable assumptions on
$h(t)$. And a kind of trichotomy to the initial hypersurface was
used by Chou-Wang \cite {cw} in logarithmic Gauss curvature flow.

This paper is organized as follows: Section 2 introduces some known
results on curvature flow (1.1) and some preliminary facts of convex
hypersurfaces, which will be used later. In section 3, we carry out
the normalization of (1.1), and estimate the inner and outer radii
of the rescaled convex hypersurfaces. In terms of the limiting shape
of the scaling factor $\psi (t)$ as $t\rightarrow T_{\max}$, long
time existence and convergence of solutions to (1.1) or (1.2) are
proved in section 4, 5 and 6, separately, and therefore we complete
the proof of Theorem 1.

\section{Preliminaries}
\label{section:2}

\renewcommand{\thesection}{\arabic{section}}
\renewcommand{\theequation}{\thesection.\arabic{equation}}
\setcounter{equation}{0}

\indent

Let $M$ be a smooth hypersurface immersion in ${\mathbb R}^{n+1}$.
We will use the same notation as in \cite {h2}. In particular, for a
local coordinate system $\{x^1, \cdots, $ $x^n\}$ of $M$, $g=g_{ij}$
and $A=h_{ij}$ denote respectively the metric and second fundamental
form of $M$. Then the mean curvature and the square of the second
fundamental form are given by
$$H=g^{ij}h_{ij}, \qquad |A|^2
=g^{ij}g^{lm}h_{il}h_{jm}, $$ where $g^{ij}$ is the $(i,j)$-entry of
the inverse of the matrix $(g_{ij})$. In the sequel we will use
$\lambda _i$ to denote the $i$-th principle curvature of the
hypersurface. Throughout this paper we sum over repeated indices
from $1$ to $n$ unless otherwise indicated.

The system of (1.1) is a strictly parabolic equation for which short
time existence is well known. The gradient on $M_t$ and
Beltrami-Laplace operator on $M_t$ are denoted by $\nabla $ and
$\triangle$ respectively. As in \cite {h2,m2}, we have the following
evolution equations for various geometric quantities under the flow
(1.1)
\begin{lemma}
The following evolution equations hold for any solution to equation
(1.1)\\[-2mm]

 (i) ~~~$\frac {\partial }{\partial
t}g_{ij}=2(h-H)h_{ij}$.\\[-3mm]

(ii) ~~$\frac {\partial }{\partial t}d\mu _t=H(h-H)d\mu _t$.\\[-3mm]

 (iii) ~
$\frac {\partial }{\partial t}{\bf v}=\nabla H$.\\[-3mm]

(iv) ~~$\frac {\partial }{\partial t}h_{ij}=\triangle
h_{ij}+(h-2H)h_{ik}h_j^k+|A|^2h_{ij}$.\\[-3mm]

(v)  ~~~$\frac {\partial }{\partial t}H=\triangle H-(h-H)|A|^2$.\\[-3mm]

(vi) ~~ $\frac {\partial }{\partial t}|A|^2=\triangle |A|^2-2|\nabla
A|^2+2|A|^4-2h\emph{tr}(A^3) $.\\[2mm]
Here $d\mu _t$ is the area element of $M_t$, and
$h_i^j=h_{ik}g^{kj}$.
\end{lemma}

Since $M_0$ is strictly convex, the curvature flow (1.1) preserves
the convexity of all $M_t$ as long as the solution exists \cite
{h2,m2}.
\begin{lemma}
(i) If $h_{ij}\geq 0$ at $t=0$, then it remains so on $[0,
T_{\max})$.\\[-3mm]

(ii) If initially $H>0$ and $h_{ij}\geq \varepsilon Hg_{ij}$ for
some $\varepsilon \in (0, \frac 1n]$, then $h_{ij}\geq \varepsilon
Hg_{ij}$ remains true, with the same $\varepsilon $ on $[0,
T_{\max})$.
\end{lemma}

This leads to the following consequence of convexity \cite {h1}
\begin{lemma} If initially $H>0$ and $h_{ij}\geq \varepsilon Hg_{ij}$ for
some $\varepsilon \in (0, \frac 1n]$ then\\[-2mm]

(i) $H\emph{tr}(A^3)-|A|^4 \geq n\varepsilon ^2H^2 (|A|^2-\frac
1nH^2)$.\\[-3mm]

(ii) $|H\nabla _ih_{kl}-h_{kl}\nabla _iH|^2\geq \frac 12 \varepsilon
^2H^2|\nabla H|^2$.
\end{lemma}

Let $|M|$ be the area of $M$, and $|V|$ the volume of the region
$V$ contained inside $M$. Lemma 2 implies that every solution of
(1.1) is a compact, convex hypersurface, therefore we have the
following relations between $|V|$ and $|M|$ by Aleksandrov-Fenchel
inequality and divergence theorem (see Theorem 2.3 in \cite {m2})
\begin{lemma}
Let $M$ be a compact and convex hypersurface embedded into
${\mathbb R}^{n+1}$ satisfying $H>0$ and $h_{ij}\geq \varepsilon Hg_{ij}$,
for some $\varepsilon \in (0, \frac 1n]$. Then there exists a
constant $c_1$ depending on $n$ and $\varepsilon$ such that
$$c_1^{-1}|M|^{\frac {n+1}{n}}\leq |V|\leq c_1|M|^{\frac {n+1}{n}}.$$
\end{lemma}

In order to study (1.1), the following facts of convex hypersurfaces
will be used.

Recall that the second fundamental form of a convex hypersurface
$X:M^n\longrightarrow {\mathbb R}^{n+1}$ is positive definite,
and the outer unit normal vector field ${\bf v}$ to the
hypersurface defines the Gauss map ${\bf v}: M^n\longrightarrow
{\mathbb S}^n$. Since the hypersurface is convex and compact,
i.e. the Gauss map is everywhere non-degenerate, we use the Gauss
map to reparametrize the convex hypersurface (see \cite {an,u,z})
$$X=X({\bf v}^{-1}(z)), \quad z\in {\mathbb S}^n.$$
Then the support function is defined as
$$\mathcal{Z}(z)=\langle z, X({\bf v}^{-1}(z))\rangle, \quad z\in
{\mathbb S}^n.$$

If we denote by $\overline{\nabla}$ and $\overline{g}$ the covariant
derivative and standard metric on ${\mathbb S}^n$, the
hypersurface can be represented by the support function
$$X(z)=\mathcal{Z}(z)z+\overline {\nabla}\mathcal{Z}(z).$$
The second fundamental form now can be calculated directly from the
support function as follows
\begin{equation}
h_{ij}=\overline {\nabla}_i\overline {\nabla
}_j\mathcal{Z}+\mathcal{Z}\overline g_{ij} \quad \mbox{ on }
{\mathbb S}^n,
\end{equation}
and the metric is given by
\begin{equation}
g_{ij}=h_{ik}\overline g^{kl}h_{lj}.
\end{equation}

 The width function of the hypersurface
$X$ is defined by
$$w(z)=\mathcal{Z}(z)+\mathcal{Z}(-z), \quad z\in {\mathbb S}^n.$$

In order to control the width of a convex hypersurface, we cite a
theorem of Andrews \cite {an}
\begin{lemma}
Let $M$ be a smooth, compact and convex hypersurface in ${\mathbb
R}^{n+1}$. Suppose that there exists a positive constant $c_2$ such
that $M$ satisfies the pointwise pinching estimate $\lambda
_{\max}(x)\leq c_2\lambda _{\min}(x)$, for every $x\in M$. Then the
following estimate holds
$$w_{\max}\leq c_2w_{\min},$$
where $\lambda _{\max}(x)$ and $\lambda _{\min}(x)$ are the largest
and smallest principal curvatures of $M$ at $x$ respectively, and
$w_{\max}=\max_{z\in {\mathbb S}^n}w(z)$ and $w_{\min}=\min_{z\in
{\mathbb S}^n}w(z)$.
\end{lemma}

By this lemma, a pinching estimate on the inner radius $r_{in}$
and outer radius $r_{out}$ immediately follows \cite {an}
\begin{corollary}
Let $M$ be a smooth, compact and convex hypersurface in ${\mathbb
R}^{n+1}$. Suppose that there exists a positive constant $c_2$ such
that $M$ satisfies the pointwise pinching estimate $\lambda
_{\max}(x)\leq c_2\lambda _{\min}(x)$, for every $x\in M$. Then
there exists a constant $c_3$ such that
$$r_{out}\leq c_3r_{in}.$$
\end{corollary}

 For a convex hypersurface $M^n$, we can also
parametrize it as a graph over the unit sphere ${\mathbb S}^n$
(cf. \cite {an,g}, see also \cite {z}). Let
$$\pi (x)=\frac {X(x)}{|X(x)|}:M^n\longrightarrow {\mathbb S}^n,$$
then we write the solution $M_t$ to equation (1.1) as a radial graph
\begin{equation}
X(x,t)=r(z,t)z: {\mathbb S}^n\longrightarrow {\mathbb
R}^{n+1},
\end{equation}
where $r(z,t)=|X(\pi ^{-1}(z),t)|$. We calculate the metric of
$M_t$ in terms of $r$ as
$$
g_{ij}=r^2\overline{g}_{ij}+\overline{\nabla }_ir\overline{\nabla
}_jr,
$$
and its inverse is
\begin{equation}
g^{ij}=r^{-2}\left(\overline{g}^{ij}-\frac {\overline{\nabla
}^ir\overline{\nabla }^jr}{r^2+|\overline{\nabla}r|^2 }\right).
\end{equation}
The outer unit normal vector and the second fundamental form of
$M_t$ in terms of $r$ are given respectively by
\begin{equation}
{\bf v}=\frac 1{\sqrt{r^2+|\overline{\nabla }r|^2}}(rz-\overline
{\nabla }r),
\end{equation}
and
\begin{equation}
h_{ij}=\frac 1{\sqrt{r^2+|\overline{\nabla
}r|^2}}(-r\overline{\nabla }_i\overline{\nabla
}_jr+2\overline{\nabla }_ir\overline{\nabla
}_jr+r^2\overline{g}_{ij}).
\end{equation}

\section{The Normalized Equation}
\label{section:3}

\renewcommand{\thesection}{\arabic{section}}
\renewcommand{\theequation}{\thesection.\arabic{equation}}
\setcounter{equation}{0}

\indent

The solution of the curvature flow (1.1) may shrink to a point if
$h$ is small enough (e.g. $h=0$ \cite {h1}), or expand to infinity
if $h$ is large enough (e.g. $h$ is a constant and $h>\sup_{x\in
M^n}H(x,0)$). The solution can also converge to a smooth
hypersurface, for some special initial hypersurface and $h$ (e.g.
the volume preserving mean curvature flow \cite {h2}, the surface
area preserving mean curvature flow \cite {m1}). In order to see
this, we normalize the equation (1.1) by keeping some geometrical
quantity fixed, for example as in \cite {h1} the total area of the
hypersurfaces $M_t$. As that mentioned in section 1, multiplying
the solution $X$ of (1.1) at each time $0\leq t<T_{\max}$ with a
positive constant $\psi (t)$ such that the total area of the
hypersurfaces $\widetilde{M}_t$ given by
$$\widetilde{X}(x, t)=\psi (t)X(x,t)$$
  has total area equal to
$|M_0|$, the area of $M_0$
\begin{equation}\int_{\widetilde{M}_t}d\widetilde{\mu}_t=|M_0|,
\qquad 0\leq t<T_{\max}. \end{equation}
 Then we introduce a new time variable $\tilde{t}(t)=\int
_0^t\psi^2(\tau)d\tau$, such that $\frac {\partial
\tilde{t}}{\partial t}=\psi^2$.

As in \cite {h1,an}, for a geometric quantity $P$ on $M_t$, we
denote by $\widetilde{P}$ the corresponding quantity on the rescaled
hypersurface $\widetilde{M}_{\tilde{t}}$. By direct calculation we
have
\begin{eqnarray*}
\widetilde{g}_{ij}=\psi^2g_{ij},& \widetilde{h}_{ij}=\psi
h_{ij},\\
\widetilde{H}=\psi^{-1}H,& \quad |\widetilde{A}|^2=\psi ^{-2}|A|^2,\\
d\widetilde{\mu}=\psi ^nd\mu,&\widetilde{w}=\psi w,
\end{eqnarray*}
and so on. If we differentiate (3.1) for time $t$, we obtain
\begin{eqnarray*}\psi ^{-1}\frac {\partial \psi}{\partial t}=\frac
1n \frac {\int _M(H-h)Hd\mu}{\int _{M}d\mu}=\frac 1n \theta.
\end{eqnarray*}

Now by differentiating $\widetilde{X}$ with respect to $\tilde{t}$,
we derive the normalized evolution equation for a different maximal
time interval $0\leq \tilde{t}<\widetilde{T}_{\max}$
\begin{eqnarray} \left\{\begin{array}{l}
\frac{\partial }{\partial \tilde{t}} \widetilde{X}(x,\tilde{t})
=\{\widetilde{h}(\tilde{t})-\widetilde{H}(x, \tilde{t})\}{\bf
\widetilde{v}}(x,\tilde{t})+{\frac 1n}
\widetilde{\theta}(\tilde{t})\widetilde{X}(x, \tilde{t}),
\\[2mm]
\widetilde{X}(\cdot, 0)= X_0, \end{array}\right.
\end{eqnarray}
where $\widetilde{h}=\psi ^{-1}h$, $\widetilde{\theta}=\psi
^{-2}\theta$ and $\theta$ is given by
$$\theta=-\frac {\int _{M}(h-H)Hd\mu}{\int _Md\mu}.$$

Since $M_t$ is convex, and $\widetilde{M}_{\tilde{t}}$ is just a
rescaling of $M_t$, therefore which is also convex, we can write
$M_t$ or $\widetilde{M}_{\tilde{t}}$ to be a graph over a unit
sphere as in (2.3). By (1.1), (2.4)$\sim$(2.6) we have the evolution
equation for $r(t)$
\begin{equation}
\frac {\partial r}{\partial t}=\frac
hr\sqrt{r^2+|\overline{\nabla}r|^2}+
r^{-3}\left(\overline{g}^{ij}-\frac {\overline{\nabla
}^ir\overline{\nabla }^jr}{ r^2+|\overline{\nabla}r|^2
}\right)\left(r\overline{\nabla }_i\overline{\nabla
}_jr-2\overline{\nabla }_ir\overline{\nabla
}_jr-r^2\overline{g}_{ij}\right).
\end{equation}
Then $\tilde{r}=\psi r$ satisfies the evolution equation
\begin{eqnarray}
\frac {\partial \tilde{r}}{\partial \tilde{t}}&=&\frac
{\widetilde{\theta}}{n} \tilde{r}+\frac
{\widetilde{h}}{\tilde{r}}\sqrt{\tilde{r}^2
+|\overline{\nabla}\tilde{r}|^2}\nonumber\\
&&+ \tilde{r}^{-3}\left(\overline{g}^{ij}-\frac {\overline{\nabla
}^i\tilde{r}\overline{\nabla }^j\tilde{r}}{
\tilde{r}^2+|\overline{\nabla}\tilde{r}|^2 }\right)
\left(\tilde{r}\overline{\nabla }_i\overline{\nabla
}_j\tilde{r}-2\overline{\nabla }_i\tilde{r}\overline{\nabla
}_j\tilde{r}-\tilde{r}^2\overline{g}_{ij}\right).
\end{eqnarray}

In the remainder of this section, we will estimate the outer and
inner radii of the normalized hypersurfaces $\widetilde{M}$. First
we see that since at each time the whole configuration of
$\widetilde{M}$ is only dilated by a constant factor $\psi $, the
solutions to (3.2) are compact and convex hypersurfaces, and Lemma 2
still holds. This means that
$$\widetilde{h}_{ij}\geq \varepsilon
\widetilde{H}\widetilde{g}_{ij},$$ for some $\varepsilon \in (0,
\frac 1n]$. The hypersurface $\widetilde{M}$ encloses a region
$\widetilde{V}$ of volume $|\widetilde{V}|$. Then by Lemma 4
\begin{equation}c_1^{-1}|\widetilde{M}|^{\frac {n+1}{n}}\leq
|\widetilde{V}|\leq c_1|\widetilde{M}|^{\frac {n+1}{n}}.
\end{equation}
Since $|\widetilde{V}|$ is controlled by the volume of its inner and
outer sphere
$$c_4\tilde{r}_{in}^{n+1}\leq |\widetilde{V}|\leq
c_4\tilde{r}_{out}^{n+1},$$ for a constant $c_4$, we obtain the
following estimate by the fixed total area of $\widetilde{M}$ by
(3.5)
\begin{equation}
\tilde{r}_{out}\geq c_5 \mbox { and } \tilde{r}_{in}\leq c_6,
\end{equation}
for some two positive constants $c_5$ and $c_6$.

By Corollary 1 and (3.6) we have
\begin{proposition} The lower bound of the inner radius and the
upper bound of the
outer radius of $\widetilde{M}_{\tilde{t}}$ are all uniformly
bounded, i.e.
$$c_7^{-1}\leq \tilde{r}_{in}\leq \tilde{r}_{out}\leq c_7$$
for some constant $c_7$.
\end{proposition}

 Now for any given time sequence $\{T_i\}$, $T_i\in [0, T_{\max})$, such
 that $T_i\rightarrow T_{\max}$ as $i\rightarrow \infty$, there
 corresponds to a sequence $\{\psi_i=\psi(T_i)\}$. By limiting
 theory, there exists at least one accumulation of this sequence.
 Denote by $\Lambda _i$ the minimal accumulation of the
 sequence $\{\psi_i=\psi(T_i)\}$. We define $\Lambda $ to be the
 infimum of $\Lambda _i$ for all possible sequences
$\{\psi_i=\psi(T_i)\}$, i.e.
\begin{eqnarray*}
\Lambda =\inf\left\{\Lambda _i|\Lambda _i \mbox { is the minimal
accumulation of a sequence
}\{\psi_i=\psi(T_i)\}\right.,\\\left.\mbox{ where }\{T_i\} \mbox{ is
any  sequence in }[0, T_{\max}) \mbox { such that } T_i\rightarrow
T_{\max} \mbox { as } i\rightarrow \infty\right\}.
\end{eqnarray*}

 Therefore by the method of extracting diagonal subsequences we have
  a subsequence, still denoted by $\{\psi_i=\psi(T_i)\}$, which converges
   to $\Lambda $ as $T_i\rightarrow T_{\max}$ (or $i\rightarrow
   \infty$), that is to say we have the following limit
  \begin{equation}
  \lim_{i\rightarrow \infty}\psi _i=\Lambda.
  \end{equation}

There are three cases in terms of the limit $\Lambda$: $\Lambda
=\infty$, $0<\Lambda <\infty$ and $\Lambda =0$. We will consider the
three cases separately in the sequel.

\section{Case (I) $\Lambda =\infty$}
\label{section:4}
\renewcommand{\thesection}{\arabic{section}}
\renewcommand{\theequation}{\thesection.\arabic{equation}}
\setcounter{equation}{0}
 \indent

In this section we consider the case $\Lambda =\infty$, and prove
Theorem 1(I).
 Since $\tilde{r}_{out}=r_{out}\psi$, we have by Proposition 1
$$\frac {c_7^{-1}}{\psi}\leq r_{out}\leq \frac
{c_7}{\psi}, $$ which implies that for the sequence $\{T_i\}$ in
last section (see (3.7)), we have a limit
\begin{equation}
\lim_{T_i\rightarrow T_{\max}} r_{out}(T_i)=0.
\end{equation}
By limiting theory, there exists a time $T^*<T_{\max}$ such that for
any $T_i\geq T^*$, $r_{out}(T_i)$ is less than any given positive
number $r^*$. By the assumption (1.3), $h(t)$ has a uniformly upper
bound $h^+$ on $[0, T_{\max})$ (We can always assume $h^+>0$ even in
the case of mean curvature flow, i.e. $h(t)=0$). We now choose $r^*$
is less than $ n/{h^+}$.

We follow an idea in \cite {an,z} to prove the following lemma which
implies that when $t$ is very near $T_{\max}$, $M_t$ is in fact
contracting.
\begin{lemma}
When $t\geq T^*$, the regions enclosed by the hypersurfaces $M_t$
are decreasing. Furthermore $T_{\max}<\infty$, and the solutions to
(1.1) converge uniformly to a point in ${\mathbb R}^{n+1}$ as
$t\rightarrow T_{\max}$.
\end{lemma}
\begin{proof}
Let $\partial B_{r^*}(O)$ be a sphere in ${\mathbb R}^{n+1}$
centered at the origin $O$, with radius $r^*$. Since the outer
radius of $M_{T^*}$ is less than $r^*$, without loss of generality,
we may assume that the hypersurface $M_{T^*}$ is enclosed by
$\partial B_{r^*}(O)$. Now we evolve the sphere $\partial
B_{r^*}(O)$ in terms of (1.1), the radius $r_B(t)$ satisfies
\begin{eqnarray}\left\{\begin{array}{l}\frac {dr_B(t)}{dt}
=h-\frac n{r_B(t)}\leq h^+-\frac n{r_B(t)},\quad t\geq
T^*,\\[2mm]
r_B(T^*)=r^*, \end{array} \right.
\end{eqnarray}
which yields that $r_B(t)$ is decreasing because $r^*<n/{h^+}$. Then
by containment principle, which can be easily derived from (3.3), we
see that the enclosed regions of $M_t$ are decreasing for $t\geq
T^*$.

Furthermore it can be checked that the solution to the differential
inequality (4.2) is given by
\begin{equation}
r_B(t)+\frac n{h^+} \log (n-{h^+}r_B(t))\geq h^+(t-T^*)+r^*+\frac
n{h^+}\log(n-h^+r^*),
\end{equation}
which yields the finiteness of $T_{\max}$ since the left hand side
of (4.3) is uniformly bounded for $t\geq T^*$.

By convexity in Lemma 2, the pinching estimate in Corollary 1 will
imply the uniformly convergence of solutions to (1.1) to a point if
we can show that the enclosed area of $M_t$ tends to $0$ as
$t\rightarrow T_{\max}$. If this is not true, we then can place a
small ball $B_{r_0}(x_0)$ in the region enclosed by $M_t$ for all
$t\in [T^*, T_{\max})$. Again without loss of generality we assume
$x_0$ is the origin. Then the diameter of $M_t$ is uniformly bounded
from below, and $|\overline{\nabla}r|$ is also uniformly bounded by
convexity. Therefore equation (3.3) is a uniformly parabolic
equation with bounded coefficients. Hence we can apply the standard
regularity theory of uniformly parabolic equations (cf. \cite {k} or
\cite {an,z}) to conclude that the solution to (3.3) can not be
singular at $t=T_{\max}$, which is a contradiction. Therefore
$X(\cdot, t)$ must converge to a point as $t\rightarrow T_{\max}$.
This completes the proof of the lemma.
\end{proof}
\begin{remark} (i) From the proof of Lemma 6, we see that the
containment principle implies that $r_{out}$ tends to zero, as
$t\rightarrow T_{\max}$. Therefore by Proposition 1 again, the
function $\psi(t)$ must tend to infinity as $t\rightarrow T_{\max}$,
i.e.
\begin{equation}\lim_{t\rightarrow T_{\max}}\psi(t) =\infty.\end{equation}

(ii) We can see that for $\overline h=\infty$, (1.1) is still
contracting to a point. In fact from the limit of $\psi(T_i)$ in
section 3, we see that $\Lambda $ is the smallest limit of $\psi $.
That is to say if $\Lambda =\infty$, then for any sequence
$\{T_j\}\subset [0, T_{\max})$ satisfying $T_j\rightarrow T_{\max}$
as $j\rightarrow \infty$ , $\lim _{j\rightarrow
\infty}\psi(T_j)=\infty$. Therefore similarly by Proposition 1, the
inner and outer radii of the evolving hypersurfaces all tend to zero
as $t\rightarrow T_{max}$. Then the containment principle implies
that the solutions to $(1.1)$ converge to a point as $t\rightarrow
T_{max}$ for all possible limits of $h(t)$.
\end{remark}

To understand the solution $X(\cdot, t)$ near the maximal time
$T_{\max}$, we consider the solution of the rescaled equation (3.2).
We want to bound the curvature $\widetilde{H}$ of
$\widetilde{M}_{\tilde{t}}$, for this purpose, we will use a trick
of Chow (Tso) \cite {t} (see also \cite {an,m2,z}) to consider the
function
\begin{equation}\Phi=\frac {H}{\mathcal{Z}-\alpha},
\end{equation}
for a constant $\alpha $ to be chosen later. First we compute the
evolution equation of $\Phi$.
\begin{lemma}
For $t\in[0, T_{\max})$, for any constant $\alpha$ we have
\begin{eqnarray}
\frac {\partial }{\partial t}\Phi &=&
g^{ij}\overline{\nabla}_i\overline{\nabla}_j\Phi+\frac
2{\mathcal{Z}-\alpha}g^{ij}\overline{\nabla }_i\Phi \overline{\nabla
}_j\mathcal{Z}\nonumber\\
&&+\frac 1{(\mathcal{Z}-\alpha)^2}\left\{2H^2-hH-\alpha
H|A|^2-h(\mathcal{Z}-\alpha)|A|^2\right\}.
\end{eqnarray}
\end{lemma}
\begin{proof} The proof is just the one in \cite {m2}. Because we shall
consider the evolution equations of similar functions in section 5
and 6, we outline its proof here. We first have
$$\overline {\nabla
}_i\Phi=\frac{\overline{\nabla}_iH}{\mathcal{Z}-\alpha}-\frac{H\overline
{\nabla }_i\mathcal{Z}}{(\mathcal{Z}-\alpha)^2},$$ and
\begin{eqnarray*}
\overline {\nabla }_i\overline {\nabla }_j\Phi=\frac{\overline
{\nabla }_i\overline{\nabla}_jH}{\mathcal{Z}-\alpha}-\frac{\overline
{\nabla }_iH\overline {\nabla }_j\mathcal{Z}+\overline {\nabla
}_i\mathcal{Z}\overline {\nabla
}_jH}{(\mathcal{Z}-\alpha)^2}-\frac{H\overline {\nabla
}_i\overline{\nabla}_j\mathcal{Z}}{(\mathcal{Z}-\alpha)^2}+\frac{2H\overline
{\nabla
}_i\mathcal{Z}\overline{\nabla}_j\mathcal{Z}}{(\mathcal{Z}-\alpha)^3},
\end{eqnarray*}
which yields
\begin{equation}g^{ij}\overline {\nabla }_i\overline
{\nabla }_j\Phi=\frac{g^{ij}\overline {\nabla
}_i\overline{\nabla}_jH}{\mathcal{Z}-\alpha}-\frac {2g^{ij}\overline
{\nabla
}_i\Phi\overline{\nabla}_j\mathcal{Z}}{\mathcal{Z}-\alpha}-\frac
{Hg^{ij}\overline {\nabla
}_i\overline{\nabla}_j\mathcal{Z}}{(\mathcal{Z}-\alpha)^2}.
\end{equation}

By differentiating the support function with respect to time $t$ we
have
$$\frac {\partial \mathcal{Z}}{\partial t }=h-H.$$

 By using (2.2), one has
\begin{eqnarray*}
H=g^{ij}h_{ij}=\overline{g}_{ij}(h^{-1})^{ij},
\end{eqnarray*}
where $(h^{-1})^{ij}$ is the inverse of $h_{ij}$. Thus by (2.1) we
have the evolution equation of $H$ in terms of the connection on
${\mathbb{S}}^n$
\begin{eqnarray*}
\frac {\partial H}{\partial t}=g^{ij}\left[\overline {\nabla
}_i\overline{\nabla}_jH+(H-h)\overline{g}_{ij}\right].
\end{eqnarray*}
Then the time derivative of $\Phi$ is given by
\begin{equation}\frac{\partial \Phi }{\partial t}=
\frac {g^{ij}}{\mathcal{Z}-\alpha}\left[\overline {\nabla
}_i\overline{\nabla}_jH+(H-h)\overline{g}_{ij}\right]-\frac
{H(h-H)}{(\mathcal{Z}-\alpha)^2}.
\end{equation}
Now by (2.2) again, we have the identity
$g^{ij}\overline{g}_{ij}=|A|^2$. Therefore by combining (4.7) and
(4.8), we obtain the expression
\begin{eqnarray*}
\frac{\partial \Phi }{\partial
t}&=&g^{ij}\overline{\nabla}_i\overline{\nabla}_j\Phi+\frac
2{\mathcal{Z}-\alpha}g^{ij}\overline{\nabla }_i\Phi
\overline{\nabla }_j\mathcal{Z}\\
&&+\frac {Hg^{ij}\overline {\nabla
}_i\overline{\nabla}_j\mathcal{Z}}{(\mathcal{Z}-\alpha)^2}-\frac
{h-H}{\mathcal{Z}-\alpha}|A|^2-\frac
{H(h-H)}{(\mathcal{Z}-\alpha)^2}\\
&=&g^{ij}\overline{\nabla}_i\overline{\nabla}_j\Phi+\frac
2{\mathcal{Z}-\alpha}g^{ij}\overline{\nabla }_i\Phi
\overline{\nabla }_j\mathcal{Z}\\
&&+\frac{1}{(\mathcal{Z}-\alpha)^2}\left\{2H^2-hH-\alpha
H|A|^2-h(\mathcal{Z}-\alpha)|A|^2\right\},
\end{eqnarray*}
which establishes the lemma.
\end{proof}

For $t\in [0, T^*]$, $M_t$ is smooth, compact and convex, and
therefore the mean curvature $H$ is uniformly bounded in this time
interval. Similarly, the mean curvature of $\widetilde{M}$ is also
bounded in the corresponding time interval. Moreover we can prove
the following
\begin{lemma}
There exists a positive constant $c_8$ such that for any
$\tilde{t}\in [0, \widetilde{T}_{\max})$,
$$\widetilde{H}(x, \tilde{t})\leq c_8, \quad \forall x\in M^n.$$
\end{lemma}
\begin{proof}
 Let $\widetilde{T}^*=\int _0^{T^*}\psi^2
(t)dt$. For any $\tilde{t}\in[0, \widetilde{T}^*]$,
$\widetilde{M}_{\tilde{t}}$ is a smooth, compact and convex
hypersurface, the mean curvature $\widetilde{H}$ is therefore
uniformly bounded in $[0, \widetilde{T}^*]$.

Consider any time $t_0\in [T^*, T_{\max})$, and choose the origin of
${\mathbb R}^{n+1}$ to be the center of the sphere of radius
$r_{in}(t_0)$, which is enclosed by $X(\cdot, t_0)$. By Lemma 6, on
the time interval $[T^*, t_0]$, the support function satisfies
$$\mathcal{Z}=\langle X, {\bf v}\rangle \geq r_{in}(t_0).$$

Let $\alpha =\frac 12r_{in}(t_0)$, we consider the function
$\Phi(z,t)$ defined in (4.5) for any $(z, t)\in {\mathbb{S}
}^n\times [T^*, t_0]$. Let $(z_1, t_1)\in {\mathbb{S}}^n\times
[T^*, t_0]$ be such that $\Phi$ achieves the maximum $\sup\{\Phi
(z,t)|(z,t)\in {\mathbb{S}}^n\times [T^*, t_0]\}$. If $t_1=T^*$,
we are done, since in this case, $H(z, t_0)\leq constant$. Thus we
may assume $t_1>T^*$, then by Lemma 7, at $(z_1, t_1)$
$$2H^2-hH-\alpha
H|A|^2-h(\mathcal{Z}-\alpha)|A|^2\geq 0. $$ We use $|A|^2\geq \frac
1nH^2$ and $\mathcal{Z}\geq 2\alpha$ to obtain
$$H(z_1, t_1)\leq \frac {2n}{\alpha}.$$
Therefore for any $z\in {\mathbb{S}}^n$,
$$\Phi(z, t_0)=\frac {H(z, t_0)}{\mathcal{Z}(z, t_0)-
\alpha}\leq \Phi (z_1, t_1),$$
which implies
$$H(z, t_0)\leq \frac {c_9}{r_{in}(t_0)},$$
for a constant $c_9$, where we have used Corollary 1. By combining
with Proposition 1, we have
$$\widetilde{H}(z, \tilde{t}_0)\leq c_{10},$$
for all $z\in {\mathbb{S}}^n$. Here $\tilde{t}_0=\int
_0^{t_0}\psi^2 (t)dt$.

Since $t_0\in [T^*, T_{\max})$ is arbitrary, $\tilde{t}_0\in
[\widetilde{T}^*, \widetilde{T}_{\max})$ is also arbitrary, we thus
have the uniform bound on $\widetilde{H}$ in $[\widetilde{T}^*,
\widetilde{T}_{\max})$. Combination with the bound in $[0,
\widetilde{T}^*]$, we at last arrive at the inequality
$\widetilde{H}(x, \tilde{t})\leq c_8$, for a constant $c_8$.
\end{proof}

We can now prove the following long time existence of (3.2). In
section 3, we have bounded the inner radius and the outer radius for
$\widetilde{X}(\cdot, \tilde{t})$, and in above, we have bounded the
speed of the equation (3.2). Thus there is a positive constant
$\delta >0$ such that for each $\tilde{t}_0\in
[0,\widetilde{T}_{\max} )$, we can write the solution
$\widetilde{X}(\cdot, \tilde{t})$ to (3.2) on the time interval
$[\tilde{t}_0,\tilde{ t}_0+\delta]$ as a graph for some $\delta
>0$
$$\widetilde{X}(z, \tilde{t})=\tilde{r}(z,\tilde{t})z, \quad z\in
{\mathbb S}^n$$ for some chosen origin, and satisfies $0<c_7^{-1}\leq
\tilde{r}(z,\tilde{t})\leq c_7$, on ${\mathbb{S}}^n\times
[\tilde{t}_0, \tilde{t}_0+\delta]$. By the convexity of all
evolving hypersurfaces, we know that $\overline{\nabla }\tilde{r}$
is also uniformly bounded. We write down the evolution equation of
$\tilde{r}$, similar to (3.4), we know that it is uniformly
parabolic. So we can use the the standard regularity theory of
uniformly parabolic equations to bound the derivatives and all
higher order derivatives of $\tilde{r}$ ( see \cite {k} or \cite
{an,z}). Hence we have proved
\begin{lemma}
$\widetilde{T}_{\max}=\infty$, and $\widetilde{M}_{\tilde{t}}$
converges to a smooth hypersurface $\widetilde{M}_{\infty}$, as
$\tilde{t}\rightarrow \infty$.
\end{lemma}
\begin{remark}
By convexity the zero order estimate of $\widetilde{A}$ follows from
Lemma 8, then one can use the induction argument as in \cite {ha}
and \cite {h1,h2} to show that the curvature derivatives
$|\widetilde{\nabla} ^m\widetilde{A}|^2$ are each bounded by a
corresponding constant $C_m(n, M_0)$ for any $m\geq 1$, since the
terms containing $\widetilde{h}$ in the evolution equation can be
easily controlled. This in turn can also imply the long time
existence of (3.2).
\end{remark}

It remains to show that the limiting hypersurface
$\widetilde{M}_{\infty}$ is a round sphere. For this purpose, we
define a function
$$\tilde{f}=\frac{|\widetilde{A}|^2}{\widetilde{H}^2}.$$
It is easy to see that $\tilde{f}$ is a scaling invariant and we
have the following lemma similar as in (\cite{m2})
\begin{lemma}We have the following evolution equation
\begin{eqnarray}
\frac{\partial}{\partial\tilde{t}}\tilde{f}
&=&\widetilde{\triangle}\tilde{f}
+\frac{2}{\widetilde{H}}\langle\widetilde\nabla_l\tilde{f},
\widetilde{\nabla}_l\widetilde{H}\rangle\nonumber\\
&&-\frac{2}{\widetilde{H}^{4}}|\widetilde{H}\widetilde{\nabla}_{l}
\widetilde{h}_{ij}
-\widetilde{h}_{ij}\widetilde{\nabla}_{l}\widetilde{H}|^{2}
-\frac{2\widetilde{h}}{\widetilde{H}^{3}}(\widetilde{H}\emph{tr}
(\widetilde{A}^{3})-|\widetilde{A}|^{4}).
\end{eqnarray}
\end{lemma}
\begin{proof} First we have the evolution equation of
$f=\frac {|A|^2}{H^2}$ (cf. \cite {m2})
\begin{equation}\frac{\partial}{\partial t}f =\triangle f
+\frac2H\langle\nabla_lf,
\nabla_lH\rangle-\frac{2}{H^4}|H\nabla_lh_{ij}-h_{ij}\nabla_lH|^2
-\frac{2h}{H^3}(H\textrm{tr}(A^3)-|A|^4).\end{equation} Therefore we
have
\begin{eqnarray*} \frac{\partial}{\partial\tilde{t}}\tilde{f}
&=&\frac{\partial}{\partial t}(\frac{|A|^2}{H^2})
\cdot\frac{\partial t}{\partial\tilde{t}}\\
&=&\left\{\triangle(\frac{|A|^2}{H^2})
+\frac2H\langle\nabla_l(\frac{|A|^2}{H^2}),  \nabla_lH\rangle\right.\\
&&\left.-\frac{2}{H^4}|H\nabla_lh_{ij}-h_{ij}\nabla_lH|^2
-\frac{2h}{H^3}(H\textrm{tr}(A^3)-|A|^4)\right\}\cdot\psi^{-2},
\end{eqnarray*}
which implies the desired equality.
\end{proof}

We then can prove the first part of Theorem 1.

\begin{proof} Recalling Lemma 3 we have by Lemma 10,
$$(\frac{\partial}{\partial\tilde{t}}-\widetilde{\triangle})\tilde{f}
\leq\frac{2}{\widetilde{H}}\langle\widetilde\nabla_l\tilde{f},
\widetilde{\nabla}_l\widetilde{H}\rangle.$$ By the weak maximum
principle,
$$\max_{\widetilde{M}_{\tilde{t}}}\tilde{f}\leq
\max_{\widetilde{M}_0}\tilde{f}.$$
Furthermore, by the strong maximum principle, if the maximum is
attained at some $(x,\tilde{t}_{0})$, $\tilde{t}_{0}>0$, then
$\tilde{f}$ is identically constant. Substituting into (4.9)
yields
$$\frac{2}{\widetilde{H}^{4}}|\widetilde{H}\widetilde{\nabla}_{l}
\widetilde{h}_{ij}
-\widetilde{h}_{ij}\widetilde{\nabla}_{l}\widetilde{H}|^{2}
+\frac{2\widetilde{h}}{\widetilde{H}^{3}}(\widetilde{H}\textrm{tr}
(\widetilde{A}^{3})-|\widetilde{A}|^{4})\equiv0.$$ Now,
$\widetilde{H}\textrm{tr}
(\widetilde{A}^{3})-|\widetilde{A}|^{4}\equiv0$ implies by Lemma 3
that
$$|\widetilde{A}|^{2}-\frac{1}{n}\widetilde{H}^{2}\equiv0,$$
i.e.
$$\sum\limits_{i<j}(\widetilde{\lambda}_{i}-
\widetilde{\lambda}_{j})^2\equiv0,$$
so at any point of $\widetilde{M}_{\tilde t}$, all the principal
curvatures are equal. Also
$|\widetilde{H}\widetilde{\nabla}_{l}\widetilde{h}_{ij}
-\widetilde{h}_{ij}\widetilde{\nabla}_{l}\widetilde{H}|^{2}\equiv0$
implies $\widetilde{\nabla}\widetilde{H}\equiv0$ by Lemma 3 (ii),
which then implies $\widetilde{\nabla}\widetilde{A}\equiv0$, so
$\widetilde{M}_{\tilde t_0}$ is a sphere. Therefore we have showed
that the function $\max\limits_{\widetilde{M}_{\tilde t}}\tilde{f}$
is strictly decreasing unless $\widetilde{M}_{\tilde{t}}$ is a
sphere. This implies that $\widetilde{M}_{\tilde{t}}$ approaches a
sphere as $\tilde t\rightarrow \infty$. Of course
$\widetilde{M}_\infty$ has the same total area $|M_0|$. Therefore
the proof of Theorem 1(I) is completed.
\end{proof}

\begin{remark}
(i)  One can use a similar method as in \cite {an,h1} to prove that
$\widetilde{M}_{\tilde{t}}$ converges to a sphere exponentially.\\[1mm]
(ii) It is easy to check that $0\leq h<\inf_{x\in M^n} H(x, 0)$ is
of this case, and $T^*$ below $(4.1)$ is equal to zero.
\end{remark}

\section{Case (II) $0<\Lambda <\infty$}
\label{section:5}

\renewcommand{\thesection}{\arabic{section}}
\renewcommand{\theequation}{\thesection.\arabic{equation}}
\setcounter{equation}{0}
 \indent

 In this section we consider the case $0<\Lambda <\infty$ and prove
the main Theorem 1(II). Since $\tilde{r}_{out}=r_{out}\psi$ and
$\tilde{r}_{in}=r_{in}\psi$, we have by Proposition 1
$$\frac {c_7^{-1}}{\psi}\leq r_{in}\leq r_{out}\leq \frac
{c_7}{\psi}, $$ which implies for the sequence $\{T_i\}$ in section
3, there exists a time $T^*<T_{\max}$ such that for any $T_i\geq
T^*$,
\begin{equation}
c_{12}^{-1}\leq r_{in}(T_i)\leq r_{out}(T_i)\leq c_{12}
\end{equation}
for some constant $c_{12}$. The following lemma shows that the inner
and outer radii of all evolving hypersurfaces $M_t$ are uniformly
bounded from below and above.
\begin{lemma}
There exists a constant $c_{13}$ such that $$c_{13}^{-1}\leq
r_{in}(t)\leq r_{out}(t)\leq c_{13}, \qquad \mbox {for any } t\in
[0, T_{\max}).
$$
\end{lemma}
\begin{proof}
We only prove the upper bound, the lower bound is similar. First
we claim that $\overline h>0$ in this case, where $\overline h$ is
the limit in (1.3). Suppose not, we can take any $h^+ >0$, such
that there exists a time $T'<T_{\max}$ and $h(t)<h^+$ for any
$t\in [T', T_{\max})$. Then by similar proof as in Lemma 6, we
prove that $M_t$ is contracting for $t\geq T'$. Therefore
$r_{out}(T_i)\rightarrow 0$ as $T_i\rightarrow T_{\max}$, which is
a contradiction to (5.1). The claim follows.

From the claim we know that there must exist a time $T'\in (T^*,
T_{\max})$ such that for any $t\in [T', T_{\max})$, $h(t)$ has a
positive lower bound $h^-> 0$.

Since $M_t$ for any $t\in [0, T']$ is smooth, compact and convex,
the corresponding outer radius is uniformly bounded from above in
this time interval. Suppose there is a time $T''>T'$ such that
$r_{out}(T'')>c_{13}$. By Corollary 1 we can assume $c_{13}$ is
large enough so that $r_{in}(T'')>\frac n{h^-}$. Again, we evolve a
sphere $\partial B_{r_{in}(T'')}(O)$ under (1.1). The solution
$r_{B}(t)$ to the differential inequality
\begin{eqnarray*}\left\{\begin{array}{l}
\frac {dr_B(t)}{dt}=h-\frac n{r_B(t)}\geq h^-
-\frac n{r_B(t)},\quad t\geq T'',\\
[3mm]r_B(T'')=r_{in}(T'')>\frac n{h^-}, \end{array} \right.
\end{eqnarray*}
is given by
\begin{eqnarray*}
r_B(t)+\frac n{h^-} \log ({h^-}r_B(t)-n)&\geq&
h^-(t-T'')+r_{in}(T'')\\
&&+\frac n{h^-}\log(h^-r_{in}(T'')-n).
\end{eqnarray*}

Clearly $r_{B}(t)\rightarrow \infty$ as $t\rightarrow \infty$. On
the other hand, by containment principle, $\partial B_{r_B(t)}(O)$
is enclosed by $M_t$ for any $t\geq T''$, since $M_{T''}$ encloses
$\partial B_{r_B(T'')}(O)$. Therefore there exists some $T_i>T''$
such that $r_{out}(T_i)\geq r_{B}(T_i)>c_{12}$, which is a
contradiction to (5.1). Combining the case in $[0, T']$, we finish
the proof of the lemma.
\end{proof}
\begin{remark}
Similar as in Remark 3, by Lemma 11 and that the hypersurface $M_t$
uniformly converges to a round sphere (see below for the proof), we
have a limit
\begin{equation}
\lim_{t\rightarrow T_{\max}}\psi (t)=\Lambda.
\end{equation}
\end{remark}

Based on a theorem of Chow and Gulliver \cite {cg}, we have as in
\cite{m1,m2} by Lemma 11 and 4,
\begin{lemma}
There is a $d=d(M_0)$ such that $M_t\subset B_d(O)$ for all $t\in
[0, T_{\max})$.
\end{lemma}

The following lemma also follows from McCoy \cite {m2}
\begin{lemma}
If $B_{4\alpha}(p_0)\subset V_{t_0}$ for some $t_0\in [0, T_{\max})$
and a point $p_0\in {\mathbb R}^{n+1}$, then
$B_{2\alpha}(p_0)\subset V_{t}$
for any $t\in [t_0, t_0+\min(\frac {6\alpha^2}{n}, T_{\max}))$.
\end{lemma}

Similar as in section 4, we consider the function $\Phi$ defined in
(4.5) for  $t\in [t_0, t_0+\min(\frac {6\alpha^2}{n}, T_{\max}))$,
and $\alpha =\frac  14c_{13}^{-1}$, where $c_{13}$ is given in Lemma
11. By using the same method as in \cite {m2}, we obtain the uniform
upper bound of the evolving mean curvature $H$.
\begin{lemma}
There exists a constant $c_{14}$ such that for any $t\in [0,
T_{\max})$
$$H(x,t)\leq c_{14}, \quad \forall x\in M^n.$$
\end{lemma}

Again by the standard regular theory of parabolic equations as in
section 4, or the argument as in \cite {h2,m1,m2}, we have
\begin{lemma}
$T_{\max}=\infty$, and $M_t$ converges to a smooth hypersurface
$M_{\infty}$, as $t\rightarrow \infty$.
\end{lemma}

Now we can prove the second part of Theorem 1.

\begin{proof}  We again
consider the function $f=\frac {|A|^2}{H^2}$. By the evolution
equation of $f$ in (4.10) and Lemma 3, similar to the proof of
Theorem 1(I), we have that $\max\limits_{M_t}f$ is strictly
decreasing unless $M_t$ is a sphere. This finishes the proof of
Theorem 1(II).
\end{proof}
\begin{remark}
(i) One can also prove that $M_t$ converges to a sphere
exponentially as in \cite {h2,m1}.\\[1mm]
(ii) By the limit (5.2), we easily see that
$\widetilde{M}_{\tilde{t}}$ converges to a sphere of total area
$|M_0|$.
\end{remark}
\section{Case (III) $\Lambda =0$}
\label{section:6}
\renewcommand{\thesection}{\arabic{section}}
\renewcommand{\theequation}{\thesection.\arabic{equation}}
\setcounter{equation}{0}
 \indent

This section is devoted to discuss the case $\Lambda =0$ and prove
the main Theorem 1(III). Similar to section 4, we have a limit
\begin{eqnarray}
\lim_{T_i\rightarrow T_{\max}} r_{in}(T_i)=\infty.
\end{eqnarray}
Then there exists a time $T^*<T_{\max}$ such that for any $T_i\geq
T^*$, $r_{in}(T_i)$ is greater than any given positive number $N$.
As before we evolve $\partial B_{r_{in}(T^*)}(O)$ and $\partial
B_{r_{out}(T^*)}(O)$ under (1.1), respectively. That is to say, they
satisfy the following equation
\begin{eqnarray}\frac
{dr_B(t)}{dt}=h(t)-\frac {n}{r_B(t)},\quad t\geq T^*,
\end{eqnarray}
with initial data $r_{in}(T^*)$ and $r_{out}(T^*)$ respectively.

First we consider the case $\overline h=0$. Integrating (6.2) from
$T^*$ to $T_i$ and using integral mean-value theorem, the outer
radius $r^+_B(t)$ of $M_t$ satisfies
\begin{equation}
r^+_B(T_i)-r_{out}(T^*)=\left[h(t_i)-\frac{n}{r^+_B(t_i)}\right]
(T_i-T^*),
\end{equation}
where $t_i\in [T^*, T_i]$.

If we suppose $T_{\max}<\infty$, and take limits of both sides in
(6.3), we have $\lim_{t\rightarrow T_{\max}}h(t)=\infty$, which
contradicts to $\overline h=0$. So $T_{\max}=\infty$.

Next we consider the case $0<\overline h<\infty$. In this case, we
choose $N$ greater than $\frac{n}{h^-}$ (now $h^-$ is the uniform
positive lower bound of $h(t)$ in $[T^*, T_{\max}$)). Therefore by
(6.2), the inner radius $r^{-}_{B}(t)$ and outer radius
$r^+_{B}(t)$ of $M_t$ satisfy the following inequalities,
respectively
\begin{eqnarray}
r^-_B(t)+\frac n{h^-} \log (h^-r^-_B(t)-n)&\geq&
h^-(t-T^*)+r_{in}(T^*)\nonumber \\
&&+\frac n{h^-}\log(h^-r_{in}(T^*)-n),
\end{eqnarray}
and \begin{eqnarray} r^+_B(t)+\frac n{h^-} \log
(h^-r^+_B(t)-n)&\geq& h^-(t-T^*)+r_{out}(T^*)\nonumber\\
&&+\frac n{h^-}\log(h^-r_{out}(T^*)-n).
\end{eqnarray}
\begin{lemma} When $t\geq T^*$, the regions enclosed by the
hypersurfaces $M_t$ are increasing. Furthermore $T_{\max}=\infty$,
and the solutions to (1.1) expand uniformly to $\infty$ as
$t\rightarrow \infty$.
\end{lemma}
\begin{proof}
For $t\geq T^*$, (6.2) implies that $r_B(t)$ is increasing since
$r_B(t)>\frac n{h^-}$ initially. By containment principle again,
the enclosed regions of $M_t$ are increasing. Moreover, all
$M_t{\;}'s$ are contained in the regions between $\partial
B_{r^-_{B}(t)}(O)$ and $\partial B_{r^+_{B}(t)}(O)$ for every
$t\in [T^*, T_{\max})$.

Suppose $T_{\max}$ is finite. Integrating Equation (6.2) from
$T^*$ to $t$, we have
\begin{eqnarray*}
r^+_B(t)-r_{out}(T^*)=\int^t_{T^*}h(\tau)d\tau-
\int^t_{T^*}\frac{n}{r^+_B(\tau)}d\tau,
\end{eqnarray*}
which implies that $r^+_B(t)$ is uniformly bounded from above in
$[T^*, T_{\max})$. This is a contradiction to (6.1). Therefore
$T_{\max}=\infty$.

 Obviously $r(z,t)\rightarrow \infty$ for any $z\in {\mathbb S}^n$
 as $t\rightarrow \infty$ by (6.4), (6.5) and the containment
 principle, which implies that $M_t$ expands to $\infty$ in this
 case. The lemma follows.
\end{proof}
\begin{remark}
Lemma 16 and Proposition 1 imply the limit
\begin{eqnarray*}
\lim_{t\rightarrow \infty}\psi(t)=0.
\end{eqnarray*}
\end{remark}

We don't know whether the rescaled mean curvature $\widetilde{H}$ is
uniformly bounded from above or not, but we can prove that if the
rescaled hypersurface $\widetilde{M}_{\tilde{t}}$ converges to a
smooth hypersurface, it must be a sphere. To this end, we need to
estimate the lower bound of the rescaled mean curvature. Again we
consider the function
$$\Phi=\frac {H}{\beta-\mathcal{Z}}$$
for some constant $\beta$. As in Lemma 7 we have the evolution
equation of $\Phi$
\begin{lemma} For $t\in [0, \infty)$ and $z\in {\mathbb S}^n$,
\begin{eqnarray*}
\frac {\partial }{\partial t}\Phi &=&
g^{ij}\overline{\nabla}_i\overline{\nabla}_j\Phi-\frac
2{\beta-\mathcal{Z}}g^{ij}\overline{\nabla}_i\Phi
\overline{\nabla}_j\mathcal{Z}\nonumber\\
&&+\frac
1{(\beta-\mathcal{Z})^2}\left\{(\beta|A|^2+h)H-
[2H^2+h(\beta-\mathcal{Z})|A|^2]\right\}.
\end{eqnarray*}
\end{lemma}

For any $t_0\in [T^*, \infty)$, let $\beta =2r_{out}(t_0)$ in Lemma
17. Then by Lemma 16, for any $t\in [T^*, t_0]$,
$$\mathcal{Z}=<X, {\bf v}>\leq r_{out}(t_0).$$
Applying the maximum principle to the evolution equation of $\Phi$,
by the same approach as in the proof of Lemma 8 we have
\begin{lemma}
There is a positive constant $c_{15}$ such that for any $(x,
\tilde{t})\in M^n\times [0, \infty)$
$$\widetilde{H}(x, \tilde{t})\geq c_{15}.$$
\end{lemma}

At last we show that the eigenvalues of the second fundamental
form approach to each other, when $\tilde{t}\rightarrow
\widetilde{T}_{\max}$. As before we consider the function defined
in section 4
$$f=\frac {|A|^2}{H^{2}}.$$

It is easy to see that $f$ is a scaling invariant. We also have the
evolution equation of $\tilde f$ as in (4.9). By similar discussion
as in the proof of Theorem 1(I), the rescaled evolving hypersurfaces
$\widetilde{M}_{\tilde t}$ tends to a sphere as $\tilde t\rightarrow
\infty$. This finishes the proof of Theorem 1(III).

\end{document}